\documentclass[11 pt]{amsart}
\usepackage[applemac]{inputenc}

\usepackage[dvipsnames,svgnames,x11names,hyperref]{xcolor}
\definecolor{ffmblue}{HTML}{006092}

\usepackage[hyphens]{url}
\usepackage[colorlinks=true,linkcolor=Maroon,citecolor=Maroon,urlcolor=ffmblue,hypertexnames=false,linktocpage]{hyperref}
\usepackage{bookmark}
\usepackage{amsmath,thmtools,mathtools}
\mathtoolsset{showonlyrefs=true}
\newcounter{mparcnt}

\usepackage{fancyhdr}
\usepackage{esint}
\usepackage{enumerate}

\usepackage{pictexwd,dcpic}
\usepackage{graphicx}
\usepackage{caption}
\swapnumbers
\declaretheorem[name=Theorem,numberwithin=section]{thm}
\declaretheorem[name=Remark,style=remark,sibling=thm]{rem}
\declaretheorem[name=Lemma,sibling=thm]{lemma}
\declaretheorem[name=Proposition,sibling=thm]{prop}

\declaretheorem[style=definition,name=Definition,numbered=no]{definition}

\numberwithin{equation}{section}

\newcommand{\ti}{\tilde}

\newcommand{\bs}{\backslash}
\newcommand{\cn}{\colon}
\newcommand{\sub}{\subset}

\newcommand{\mr}{\mathring}

\newcommand{\bbR}{\mathbb{R}}

\newcommand{\bbS}{\mathbb{S}}

\newcommand{\8}{\infty}

\newcommand{\al}{\alpha}
\newcommand{\be}{\beta}
\newcommand{\ga}{\gamma}
\newcommand{\de}{\delta}
\newcommand{\ep}{\epsilon}
\newcommand{\ka}{\kappa}
\newcommand{\la}{\lambda}

\newcommand{\Om}{\Omega}
\newcommand{\De}{\Delta}

\newcommand{\cU}{\mathcal{U}}

\newcommand{\cW}{\mathcal{W}}

\newcommand{\del}{\partial}
\newcommand{\n}{\nabla}

\newcommand{\fa}{\forall}
\newcommand{\rt}{\sqrt}

\newcommand{\ip}[2]{\left\langle #1,#2 \right\rangle}
\newcommand{\fr}[2]{\frac{#1}{#2}}
\newcommand{\tfr}[2]{\tfrac{#1}{#2}}

\DeclareMathOperator{\dive}{div}
\DeclareMathOperator{\id}{id}
\DeclareMathOperator{\pr}{pr}

\DeclareMathOperator{\dist}{dist}

\DeclareMathOperator{\tr}{tr}

\DeclareMathOperator{\diam}{diam}

\newcommand{\pf}[1]{\begin{proof}#1 \end{proof}}
\newcommand{\eq}[1]{\begin{equation}\begin{alignedat}{2} #1 \end{alignedat}\end{equation}}

\newcommand{\br}[1]{\left(#1\right)}
\newcommand{\abs}[1]{\lvert #1\rvert}
\newcommand{\enum}[1]{\begin{enumerate}[(i)] #1 \end{enumerate}}

\newcommand{\ra}{\rightarrow}

\newcommand{\mt}{\mapsto}

\newcommand{\hp}{\hphantom}
\newcommand{\q}{\quad}

\begin{document}
\title[Stability for anisotropic curvature functionals]{Stability of the Wulff shape with respect to anisotropic curvature functionals}

\author{Julian Scheuer}
\address{\flushleft\parbox{\linewidth}{{\bf Julian Scheuer}\\Goethe-Universit\"at\\ Institut f\"ur Mathematik\\ Robert-Mayer-Str.~10\\ 60325 Frankfurt\\ Germany\\ {\href{mailto:scheuer@math.uni-frankfurt.de}{scheuer@math.uni-frankfurt.de}}}}

%
 \author{Xuwen Zhang}
 \address{\flushleft\parbox{\linewidth}{{\bf Xuwen Zhang}\\School of Mathematical Sciences\\ Xiamen University\\ 361005 Xiamen\\ P.R. China}}  
\address{\flushleft\parbox{\linewidth}{Goethe-Universit\"at\\ Institut f\"ur Mathematik\\ Robert-Mayer-Str.~10\\ 60325 Frankfurt\\ Germany\\ {\href{mailto:zhang@math.uni-frankfurt.de}{zhang@math.uni-frankfurt.de}} }}

\begin{abstract}
For a function $f$ which foliates a one-sided neighbourhood of a closed hypersurface $M$, we give an estimate of the distance of $M$ to a Wulff shape in terms of the $L^{p}$-norm of the traceless $F$-Hessian of $f$, where $F$ is the support function of the Wulff shape. This theorem is applied to prove quantitative stability results for the anisotropic Heintze-Karcher inequality, the anisotropic Alexandrov problem, as well as for the anisotropic overdetermined boundary value problem of Serrin-type.
\end{abstract}

\date{\today}
\maketitle
\tableofcontents

\section{Introduction}

The classical umbilicity theorem for a closed hypersurface $M$ of the Euclidean space states that if at every given point $x\in M$ all the principal curvatures coincide, then $M$ is a round sphere. From this theorem it is by today quite straightforward to deduce other such {\it rigidity theorems}, such as the classification of closed and embedded constant mean curvature hypersurfaces \cite{Alexandroff:12/1962}, the Heintze-Karcher inequality \cite{MontielRos:/1991,Ros:/1987},
\eq{\int_{M}\fr{n}{H}\geq (n+1)\abs{\Om},}
and the classification of domains that admit solutions to the torsion Dirichlet problem
\eq{\De u &= 1\q\mbox{in}~\Om\\
		 u&= 0\q\mbox{on}~\del\Om }
with constant Neumann derivative \cite{Serrin:01/1971}.
Although these results were originally all proved by other means, by today they can all be proved via more or less direct application of the umbilicity theorem. These approaches also allow the extraction of the so-called {\it corresponding stability results}, which estimate the distance to a round sphere in terms of the deficits within the rigidity assumptions in the respective problems. Such results have been intensively studied for example in \cite{BrandoliniNitschSalaniTrombetti:09/2008,Ciraolo:07/2020,CiraoloMaggi:04/2017,CiraoloVezzoni:02/2018,MagnaniniPoggesi:01/2020,MagnaniniPoggesi:/2020,ScheuerXia:10/2022}.  As in this paper we are not interested in the classical isotropic setting, we do not expand further on the history of these problems but refer to \cite{Scheuer:03/2021} for a much broader overview. Many rigidity results for hypersurfaces are proved by solving elliptic Dirichlet problems and showing that the traceless Hessian of the solution $f$ vanishes. A rigidity result due to Reilly \cite{Reilly:03/1980} then gives roundness of the level sets. In \cite{Scheuer:03/2021} the first author proved a stability version for this result of the form
\eq{\label{intro-1}\dist(M,\bbS^{n})\leq C\|\mr{\bar\n}^{2}f\|_{p,\cU}^{\fr{p}{p+1}}}
with a constant $C$ depending on various geometric quantities, $\|\cdot\|_{p,\cU}$ is the $L^{p}$-norm of some suitable one-sided neighbourhood $\cU$ of $M$ and where $p>n$, $n$ being the dimension of the hypersurface. We don't go into further details here. This estimate provides a very direct approach to the stability question of many geometric rigidity theorems.

The aim of this paper is to obtain an estimate as \eqref{intro-1} in the anisotropic setting and to give a few applications. If $\cW$ is a strictly convex, smooth hypersurface a.k.a. {\it Wulff shape} with support function $F$, one can define the so-called {\it anisotropic principal curvatures} of another hypersurface $M$ as the eigenvalues of an anisotropic second fundamental form, see \autoref{sec:Prelim} for the detailed setup. Then there is a characterization of the Wulff shape: It is up to rigid motions the only smooth hypersurface with traceless anisotropic second fundamental form, \cite{Palmer:12/1998}. Stability versions of this were recently proved in \cite{De-RosaGioffre:/2019,De-RosaGioffre:/2021}. 

First we obtain a stability result similar to \eqref{intro-1} in the anisotropic setting: The Hausdorff distance of a hypersurface $M$ to a Wulff shape is estimated in terms the the traceless $F$-Hessian operator of a  function defined in a one-sided neighbourhood of $M$ satisfying $f_{|M}=0$. The key to this estimate is a relation of the traceless anisotropic second fundamental form of the level sets and the traceless $F$-Hessian. The co-area formula in conjunction with the anisotropic almost umbilicity theorem \cite{De-RosaGioffre:/2021} then gives the result.

This result will then be used to give a unified approach to stability of the anisotropic versions of Alexandroff's theorem, the Heintze-Karcher inequality and Serrin's overdetermined problem.

To state the first main result, let us briefly skim over the notation used. Details are deferred to \autoref{sec:Prelim}. In the following theorem, besides the things which are self-explanatory, we denote by $S$ the anisotropic Weingarten operator which is basically given by
\eq{S = dN^{-1}\circ d\ti\nu,}
where $N$ and $\ti \nu$ are the Gauss maps of the Wulff shape $\cW$ and $M$ respectively. $D$ is the Euclidean Levi-Civita connection and $D^{\sharp}$ the gradient operator. The {\it pinching quantity} $\mr{\bar\n}^{2}_{F}f$ is an elliptic operator arising from the so-called {\it elliptic integrand} $F$, which may be degenerate at points where $Df=0$. $L^{p}$-norms are defined with respect to the standard induced Euclidean volume element $d\ti\mu$ and $d\mu$ is the anisotropic area measure
\eq{d\mu = F(\ti \nu)d\ti \mu.}
Any dilation of the Wulff shape $\cW$ we will call {\it Wulff sphere}. The spaces $C^{k}$ and $C^{k,\al}$ are the usual spaces of $k$-times differentiable functions resp. such with $\al$-H\"older continuous highest derivatives.
Then our first main result reads as follows.

\begin{thm}\label{Thm-levelset-stab}
Let $n\geq2$ and let $F$ be an elliptic integrand on $\mathbb{R}^{n+1}$.
Let $M\sub\bbR^{n+1}$ be a closed and connected $C^{2}$-hypersurface.
Suppose that there is a one-sided neighborhood $\mathcal{U}\subset\mathbb{R}^{n+1}$ of $M$, which is foliated by the level-sets of a function $f\in C^{2}(\overline{\mathcal{U}})$, i.e.
\eq{
\overline{\mathcal{U}}
=\bigcup_{0\leq t\leq \max_{\bar\cU}\abs{f}}M_t,\quad M_t=\{\abs{f}=t\},}
such that
$f_{|M} = 0$ and $Df\neq 0$ on $\overline{\mathcal{U}}$.
Let $p>n$ and suppose that for some $C_0\geq1$, there holds
\eq{\label{condi-hnorm-upperbound}\mu(M)^{\frac{p-n}{np}}\max_{0\leq t\leq \max_{\bar\cU}\abs{f}}\|S\|_{p,M_{t}}\leq C_0,
}	
then there exists a positive constant $C=C(n,F,p,C_0)$, such that 
\eq{\label{condi-tracefree-D^2f}\|\mr{\bar\n}^{2}_{F}f\|_{p,\mathcal{U}}^{\frac{p}{p+1}}
    < \fr{1}{C}\mu(M)^{\frac{n+1-2p}{n(p+1)}}\min\left(\mu(M)^{\fr{1}{n}}\min_{\overline{\mathcal{U}}}\vert D^\sharp f\vert,\max_{\overline{\mathcal{U}}}\abs{f}\right)^{\frac{p}{p+1}}
    }
    implies the existence of a Wulff sphere $\bf{W}$ with the property
    \eq{\label{conclu-distance}\dist(M,\mathbf{W})
    \leq \frac{C\mu(M)^{\frac{3p-n}{n(p+1)}}}{\min\left(\mu(M)^{\fr 1n}\min_{\overline{\mathcal{U}}}\vert D^\sharp f\vert,\max_{\overline{\mathcal{U}}}\abs{f}\right)^{\frac{p}{p+1}}} \|\mr{\bar\n}^{2}_{F}f\|_{p,\mathcal{U}}^{\frac{p}{p+1}}.
    }
\end{thm}

Note that this theorem is invariant under the scaling transformations
\eq{f\mt \la f,\q x\mt \la x.}
We are going to prove this theorem in \autoref{sec:level set stability}.

As a first application, we obtain a direct estimate of the distance to a Wulff sphere in terms of the anisotropic Heintze-Karcher inequality. The latter has exactly the same form as in the Euclidean space, up to replacing the mean curvature and the area element by their anisotropic counterparts, and it was proved in \cite{GeHeLiMa:/2009} and later in \cite{MaXiong:/2013} with a flow method.
It was already shown in \cite[Prop.~3.7]{DelgadinoMaggiMihailaNeumayer:07/2018} that the anisotropic torsion potential can be controlled by a Heintze-Karcher deficit. However, by using a differently defined Heintze-Karcher deficit, it is possible to give a simplified estimate on the traceless $F$-Hessian of the anisotropic torsion potential and in that way obtain a direct stability estimate via \autoref{Thm-levelset-stab}.
In the following, for $0< \al<1$, $\abs{M}_{2,\al}$ denotes a constant which depends on the $C^{2,\al}$-geometry as it appears in the Schauder estimates.

\begin{thm}\label{Thm-Sta-HK}
Let $n\geq 2$, $\al>0$ and $F$ be an elliptic integrand on $\mathbb{R}^{n+1}$.
Let $\Om\sub \bbR^{n+1}$ be a bounded domain with connected $F$-mean convex boundary $M\in C^{2,\al}$ that satisfies a uniform interior Wulff sphere condition with radius $r$. 
Then there exists a positive constant $C$ depending only on $n,\al,F,r,\mu(M)$ and $\abs{M}_{2,\al}$, such that
\eq{
\left(\int_{M}\frac{1}{H}\,d\mu-\frac{n+1}{n}\vert\Om\vert\right)
\leq C^{-1}
}
implies
\eq{
\dist(M,{\bf W})
    \leq C
    \left(\int_{M}\frac{1}{H}\,d\mu-\frac{n+1}{n}\vert\Om\vert\right)^{\frac{1}{n+2}}
}
for some Wulff sphere ${\bf W}\subset\mathbb{R}^{n+1}$.
\end{thm}

In \cite[Thm.~1.1]{DelgadinoMaggiMihailaNeumayer:07/2018} the authors use the estimates on their Heintze-Karcher deficit to obtain a qualitative bubbling analysis for the constant anisotropic mean curvature problem. Other stability results for some Weingarten type equations were obtained in \cite{RothUpadhyay:03/2019}, while a stability result for almost constant first and second order mean curvatures we given in \cite{Roth:/2018}. A direct and quantitative stability estimate in terms of the $L^{1}$-deficit seems to be missing in the literature. As a second application of \autoref{Thm-levelset-stab}, we provide such a result by giving a quantitative stability estimate (without bubbling) of a hypersurface with almost constant anisotropic mean curvature in terms of an $L^{1}$-pinching.

\begin{thm}\label{Thm-Sta-Alex}
Let $n\geq 2$, $\al>0$ and $F$ be an elliptic integrand on $\mathbb{R}^{n+1}$.
Let $\Om\sub \bbR^{n+1}$ be a bounded domain with connected boundary $M\in C^{2,\al}$ that satisfies a uniform interior Wulff sphere condition with radius $r$. 
Then there exists a positive constant $C$ depending only on $n,\al,F,r,\mu(M)$ and $\abs{M}_{2,\al}$, such that
\eq{
\left\|H-\tfr{n}{n+1}\tfr{\mu(M)}{\abs{\Om}}\right\|_{1,M}
\leq C^{-1}
}
implies
\eq{
\dist(M,{\bf W})
    \leq C
    \left\|H-\tfr{n}{n+1}\tfr{\mu(M)}{\abs{\Om}}\right\|_{1,M}^{\frac{1}{n+2}}
}
for some Wulff sphere ${\bf W}\subset\mathbb{R}^{n+1}$.
\end{thm}

Finally, we consider a third problem, namely the stability in the aniso\-tropic version of Serrin's overdetermined problem. Therefore we consider the $F$-anisotropic torsion potential, i.e. a solution to
\eq{\bar\De_{F}f &=1\q\mbox{in}~\Om\\
		f&=0,\q\mbox{on}~\del\Om.}
While it is well known that a solution to this PDE with constant $F$-Neumann derivative must be a Wulff sphere \cite{CianchiSalani:12/2009,WangXia:/2011}, we are not aware of any quantitative stability versions of this in the anisotropic setting. Here we prove the following:

\begin{thm}\label{Thm-Sta-OverD}
Let $n\geq 2$, $\al>0$ and $F$ be an elliptic integrand on $\mathbb{R}^{n+1}$.
Let $\Om\sub \bbR^{n+1}$ be a bounded domain with connected boundary $M\in C^{2,\al}$ that satisfies a uniform interior Wulff sphere condition with radius $r$. 
Then there exists a positive constant depending only on $n,\al, F,r,\diam(\Om)$, $\mu(M)$ and $\abs{M}_{2,\al}$, such that
\eq{
\left\|F(D^\sharp f)-\tfrac{\vert\Om\vert}{\mu(M)}\right\|_{1,M}
    \leq C^{-1}
}
implies
\eq{
\dist(M,{\bf W})
    \leq C
    \left\|F(D^\sharp f)-\tfrac{\vert\Om\vert}{\mu(M)}\right\|_{1,M}^{\frac{1}{2(n+2)}}
}
for some Wulff sphere ${\bf W}\subset\mathbb{R}^{n+1}$.
\end{thm}

The paper is organized as follows. In \autoref{sec:Prelim} we give some necessary background on anisotropic geometry, while in the following sections we prove the theorems listed above one by one.

\section{Anisotropic geometry}\label{sec:Prelim}

Let $\left<\cdot,\cdot\right>$ denote the Euclidean scalar product on $\mathbb{R}^{n+1}$.
Let $D,D^\sharp,{\rm div}$ denote respectively the Levi-Civita connection, gradient operator and divergence operator in $(\mathbb{R}^{n+1},\ip{\cdot}{\cdot})$.
The canonical basis in the Euclidean space is denoted by $\{e_1,\ldots,e_{n+1}\}$ and for a function $f\in C^\infty(\mathbb{R}^{n+1})$, we write a comma when taking partial derivatives, that is,
\eq{f_{,\alpha}
= e_{\alpha}f
= \frac{\partial}{\partial x_\alpha}f,
\quad
f_{,\alpha\beta}
= e_{\be}e_{\al}f
= \frac{\partial^2}{\partial x_{\alpha}\partial x_{\beta}}f.
}
$\dist(\cdot,\cdot)$ and $\diam(\cdot)$ denote the Hausdorff distance respectively the diameter of sets in $(\mathbb{R}^{n+1},\ip{\cdot}{\cdot})$.

\subsection*{Ambient geometry }
We fix a smooth, closed and strictly convex hypersurface $\cW\sub \bbR^{n+1}$, where $\cW$ encloses a strictly convex body $\cW_{0}$ in $\bbR^{n+1}$, which contains the origin. We assume $n\geq 2$ throughout the paper. In the setting we are interested in, $\cW$ is often called the {\it Wulff shape}. The Wulff shape introduces a new geometry on $\bbR^{n+1}$ via the so-called {\it Minkowski norm},\eq{F^{0}(x):=\inf\{s\in (0,\8)\cn~ x\in s\cW_{0}\},}
although the homogeneity relation only holds for positive dilations.  The Wulff shape is the unit sphere for the Minkowski norm,
\eq{\cW = \{x\in \bbR^{n+1}\cn~ F^{0}(x)=1\}.}
We also call $\cW_{r}$ a {\it Wulff sphere of radius $r$}, if it is given by
\eq{\cW_{r} = \{x\in \bbR^{n+1}\cn F^{0}(x)=r\}.}

Denote by 
\eq{F\cn \bbR^{n+1}\ra (0,\8),\q F(z)= \sup_{x\in \cW}\ip{x}{z} = \sup_{x\neq 0}\fr{\ip{x}{z}}{F^{0}(x)}
} the support function of $\cW$. This is a convex and $1$-homogenous function and will be referred to as \textit{elliptic integrand} in this paper.

An embedding of $\cW$ is given by the restriction to $\bbS^{n}$ of the $0$-homogenous map $D^{\sharp}F$, 
\eq{\Phi\cn~ \bbS^{n}\ra \cW,\q \Phi := (D^{\sharp} F)_{|\bbS^{n}}= D^{\sharp}_{\bbS^{n}} F + F\id_{\bbS^{n}},}
where
$D^{\sharp}_{\bbS^{n}}F = \pr_{\bbS^{n}}(D^{\sharp} F)$ is the gradient on the round sphere and where we used \cite[Cor.~1.7.3]{Schneider:/2014} and 
\eq{\label{F-homogeneity}\left<\Phi(z),z\right>=F(z)  \q\fa z\in\mathbb{S}^{n}.}

If we denote the minimum and the maximum values of $F$ on $\mathbb{S}^n$ by
\eq{\label{defn-m_F-M_F}m_F=\min_{\mathbb{S}^n}F,\quad M_F=\max_{\mathbb{S}^n}F,}
then we have
\eq{m_F\leq\vert D^\sharp F(z)\vert
	\leq M_F,\quad\forall z\in\mathbb{R}^{n+1}\setminus\{0\}.}

The Minkowski norm is smooth on $\bbR^{n+1}\bs\{0\}$, see \cite[Thm.~13.15]{KrieglMichor:/1997}, and hence for the associated quadratic form $q(x) = \tfr 12 (F^{0}(x))^{2}$ we may define a new Riemannian metric on $\bbR^{n+1}$ by
\eq{\bar g_{x} := D^{2}q(x),}
which is, due to the strict convexity of the norm, a metric on $\bbR^{n+1}\bs\{0\}$.

Now we collect some more useful formulae for $q$ and $F$. Due to the relation
\eq{\tfr 12= q\circ \Phi = q\circ D^{\sharp} F,}
we obtain for all $z,V\in \bbR^{n+1}$
\eq{0=D_{V}(q\circ D^{\sharp} F)(z) = Dq(\Phi(z))(D_{V}D^{\sharp} F(z))}
or in coordinates,
\eq{\label{prelim-1}0 = q_{,\al}(D^{\sharp}F(z)){F_{,}}^{\al}_{\be}(z).}
Also there hold
\eq{\label{F and F0}D^{\sharp}F^{0}(D^{\sharp}F(z)) = \fr{z}{F(z)},\q D^{\sharp}F(D^{\sharp}F^{0}(z)) = \fr{z}{F^{0}(z)},}
see \cite[Prop.~1.3]{Xia:/2012}. Differentiation of the first equation gives in Euclidean coordinates
\eq{\de^{\al}_{\be} ={(F^{0})_{,}}^{\al}_{\ga}{F_{,}}^{\ga}_{\be}F + {(F^{0})_{,}}^{\al}F_{,\be}. }
Now use 
\eq{q_{,\al} = F^{0}(F^{0})_{,\al},\q q_{,\al\be} = (F^{0})_{,\al}(F^{0})_{,\be} + F^{0}(F^{0})_{,\al\be}}
and \eqref{prelim-1} to obtain
\eq{\label{prelim-2}\de^{\al}_{\be}&= \fr{1}{F^{0}}{q_{,}}^{\al}_{\ga}{F_{,}}^{\ga}_{\be}F + \fr{1}{F^{0}}{q_{,}}^{\al}F_{,\be}\\
			 &= \fr{1}{F^{0}}{q_{,}}^{\al}_{\ga}{\br{\tfr 12 F^{2}}_{,}}^{\ga}_{\be} - \fr{1}{F^{0}}{q_{,}}^{\al}_{\ga}{F_{,}}^{\ga}F_{,\be}+ \fr{1}{F^{0}}{q_{,}}^{\al}F_{,\be}\\
			&= {q_{,}}^{\al}_{\ga}(D^{\sharp}F(z)){\br{\tfr 12 F^{2}}_{,}}^{\ga}_{\be}(z), }
where we used $D^{\sharp}F \in \cW$.

\subsection*{Induced anisotropic geometry}
Let $x$ be the embedding of closed mani\-fold $M$ into $\bbR^{n+1}$ with Gauss map $\ti\nu\cn M\ra \bbS^{n}$. The anisotropic normal is defined by
\eq{\nu:= \Phi\circ \ti\nu}
and due to the homogeneity of $q$ it has the property
\eq{\label{eq-bar-g}
\bar g_{\nu}(\nu,x_{\ast}V) = 0,\q \bar g_{\nu}(\nu,\nu) = 1}
for all $V\in T_{x}M = T_{\nu(x)}\cW$.
The {\it induced anisotropic metric} of $x$ on $M$ is defined by
\eq{g_{x}(V,W) := \bar g_{\nu(x)}(x_{\ast}V,x_{\ast}W).}
It is worthy to note that the henceforth defined Riemannian geometry is not simply the geometry induced by $\bar g$ and the embedding $x$, as the anisotropic normal enters via the ambient metric.

The {\it anisotropic second fundamental form} of $x$ on $M$ is the symmetric bilinear form
\eq{h_{ij}:=h_{x}(\del_{i},\del_{j}) := -\bar g_{\nu(x)}(\nu,x_{,ij}) = \bar g_{\nu(x)}(\nu_{,i},x_{,j}),}
where we used \eqref{eq-bar-g}
and
\eq{\label{eq: Q}Q(\nu)(\nu,X,Y):=D^{3}q(\nu)(\nu,X,Y) = 0}
for all $X,Y\in T_{\nu(x)}\cW$.
The {\it anisotropic shape operator} is
\eq{S^{i}_{j} = g^{ik}h_{kj},}
where $(g^{ij})$ is the inverse of $g_{ij} = g_{x}(\del_{i},\del_{j})$. The eigenvalues $\ka_{1}\leq \dots\leq\ka_{n}$ of $S$ are called the {\it anisotropic principal curvatures} and from the symmetry of $h_{ij}$ they are real valued.

For $r\in \{1,\cdots, n\}$, the \textit{normalized $r$-th anisotropic mean curvature} is defined by
$$H_r=\frac{1}{\binom{n}{r}}\sigma_r,$$ where $\sigma_r$ be the $r$-th elementary symmetric function on the anisotropic principal curvatures $\{\kappa_i\}_{i=1}^n$, namely,
\eq{
    \sigma_r
    =\sum_{1\leq i_1<\cdots<i_r\leq n}\kappa_{i_1}\cdots\kappa_{i_r}.
}

Define the {\it anisotropic area element} by
\eq{\label{anisotropic area}d\mu: = F(\ti\nu)d\ti\mu,}
where $d\ti\mu$ is the area element on $M$ induced by the standard Euclidean volume form, see \cite[p.~1406]{Xia:/2013}. Note that $d\mu$ is not equal to th area element induced by $g$, see \cite[Lemma~2.8]{Xia:/2013}. We also define
\eq{\mu(M):=\int_{M}d\mu}
and $\ti\mu(M)$ accordingly.
For submanifolds $M\sub \bbR^{n+1}$ and any function on an open set $\cU$ containing $M$, we define $L^{p}$-norms with respect to the standard Euclidean structures, i.e. we define
\eq{\|f\|_{p,\cU}:=\br{\int_{\cU}\abs{f}^{p}\,dx}^{\fr 1p},\q \|f\|_{p,M}:=\br{\int_{M}\abs{f}^{p}\,d\ti \mu}^{\fr 1p},}
where $dx$ is the standard Euclidean volume element.

Finally we want to comment on our use of the symbol $\abs{\cdot}$. If we plug in a gradient $\n^{G}f$ of a function $f$ on manifold $N$, taken with respect to some inner product $G$, we write
\eq{\abs{\n^{G}f}: = \rt{G^{ij}\del_{i}f\del_{j}f}.}
If we plug in a {\it real-diagonalisable} endomorphism field $A\in T^{1,1}(N)$, we can define
\eq{\abs{A}:=\rt{A^{i}_{j}A^{j}_{i}},}
both definitions avoiding the necessity to refer to a metric within the notation $\abs{\cdot}$.
In this manner, when we plug in tensors into an $L^{p}$-norm, we understand this to be
\eq{\|S\|_{p,M}:=\|\abs{S}\|_{p,M}.}

\begin{rem}
We emphasize here, that for ease of notation we do not furnish anisotropic quantities induced on $M$ by some index $F$. We do so, because in this paper we will not use the classical isotropic geometric quantities and hence this way notation looks much cleaner. If on rare occasions we use the isotropic objects, they will have a tilde. This rule only applies to induced quantities on a hypersurface $M$, like $\nu$, $g$, $h$ and $H$. It does not apply to the $F$-differential operators defined in the next subsection.
\end{rem}

\subsection*{Level-sets}
Let $M\sub \bbR^{n+1}$ be the regular level-set of a smooth function $f$, embedded by a map $x$. Then its Euclidean normal is mod sign given by
\eq{\ti \nu = \fr{D^{\sharp}f}{\abs{D^{\sharp}f}},}
while its anisotropic normal is $\nu = \Phi\circ \ti\nu$. We define the $F$-gradient of $f$ as the unique vector $\bar\n^{\sharp}_{F}f$ with the property
\eq{\bar g_{\nu(x)}(\bar\n^{\sharp}_{F}f,V) = D f(x)V\q\fa V\in T_{\nu(x)}\bbR^{n+1}.}
Since for all $V\in T_{x}M = T_{\nu(x)}\cW$ there holds
\eq{0=\bar g_{\nu(x)}(\nu,x_{\ast}V) = D f(x)V=\bar g_{\nu(x)}(\bar\n^{\sharp}_{F}f,x_{\ast}V),}
 for some $\la\in \bbR$ we have
\eq{\bar\n^{\sharp}_{F}f = \la\nu=\la D^{\sharp}F(D^{\sharp}f).}
Testing with $\bar g_{\nu}(\nu,\cdot)$ we compute
\eq{\la = \bar g_{\nu}(\bar\n^{\sharp}_{F}f,\nu) = DfD^{\sharp}F(\ti \nu) = DF(\ti \nu)\ti \nu\abs{D^{\sharp}f} = F(D^{\sharp}f).}
 Hence there holds
\eq{\bar\n^{\sharp}_{F}f = F(D^{\sharp}f)D^{\sharp}F(D^{\sharp}f). }
Now we define the {\it $F$-Hessian endomorphism}, or in short {\it $F$-Hessian}, by the formula
\eq{\bar\n^{2}_{F}f(x) := \bar g^{\ga\be}_{\nu(x)}f_{,\al\be}(x) = D^{2}\br{\tfr 12 F^{2}}(D^{\sharp}f) \circ D^{2}f \in T^{1,1}_{x}M,}
where we slightly abused notation in the term on the right hand side and used \eqref{prelim-2}.
Hence we are able to express the anisotropic second fundamental form of the level-sets in terms of the $F$-Hessian by differentiation of $c = f(x)$ along the level-set $M$:
\eq{0 = f_{,\al}x^{\al}_{,i}}
and, using \cite[Lemma~3.1]{Xia:/2017},
\eq{0=f_{,\al\be}x^{\al}_{,i}x^{\be}_{,j} + f_{,\al}x^{\al}_{,ij} = f_{,\al\be}x^{\al}_{,i}x^{\be}_{,j} - f_{,\al}\nu^{\al}h_{ij} = f_{,\al\be}x^{\al}_{,i}x^{\be}_{,j} - F(D^{\sharp}f)h_{ij}.}
Hence
\eq{\label{eq-h-D^2f}h = \fr{1}{F(D^{\sharp}f)}D^{2}f_{|T_{x}M} = \fr{1}{F(\ti \nu)}\ti h.}
Tracing with respect to $g_{x}$ gives the following formula for the anisotropic mean curvature,
\eq{F(D^{\sharp}F)H =f_{,\al\be}(\bar g_{\nu(x)}^{\al\be} - \nu^{\al}\nu^{\be}) \equiv \bar\De_{F}f-D^{2}f(\nu,\nu),}
where we also introduced the {\it $F$-Laplacian} $\bar\De_{F}f$ as the trace of the $F$-Hessian enodomorphism.
The \textit{trace-free $F$-Hessian of $f$}  is then given by 
\eq{\mr{\bar\n}^{2}_{F}f = \bar\n^{2}_{F}f - \tfr{1}{n+1}\bar\De_{F}f\id. }
A direct computation shows that
\eq{\label{eq-tracefree-D^2f}
	 \vert \mathring{\bar\n}_{F}^2f\vert^2
	 &= \bar g^{\alpha\beta}_\nu\bar g^{\gamma\tau}_\nu\left(f_{,\al\ga}-\frac{\bar\De_Ff}{n+1}(\bar g_{\nu})_{\al\ga}\right)\left(f_{,\be\tau}-\frac{\bar\De_Ff}{n+1}(\bar g_{\nu})_{\be\tau}\right)\\
	 &= \tr\left((\bar\nabla^2_F f)^2\right)-\frac{(\bar\De_Ff)^2}{n+1}\\
	 & = \abs{\bar\n^{2}_{F}f}^{2} - \fr{(\bar\De_{F}f)^{2}}{n+1}.
}

\section{Level-set stability}\label{sec:level set stability}

In this section we prove a theorem which controls the distance of a hypersurface to the Wulff shape by a level set foliation of a one-sided neighbourhood of that hypersurface. The analogous result in the isotropic setting was proved by the first author in \cite{Scheuer:03/2021}.

\begin{proof}[Proof of \autoref{Thm-levelset-stab}]
We modify the proof of \cite[Thm.~1.1]{Scheuer:03/2021} where necessary.
We prove the case when $\mu(M)=1$, the general case follows after rescaling. Assume $f_{|\cU}>0$, otherwise consider $-f$ instead. Since $M_{t} = \{f=t\}$ are regular level sets of $f$, at $M$ the interior pointing Euclidean normal is given by
\eq{\ti \nu = \fr{D^{\sharp}f}{{\abs{D^{\sharp}f}}},}
while the corresponding anisotropic normal is 
\eq{\nu = \Phi(\ti \nu) = D^{\sharp}F(D^{\sharp}f).}
		
Thanks to \eqref{eq-h-D^2f}, by defining $\mathring{D}^2f\coloneqq D^2f-\frac{\bar\De_Ff}{n+1}\bar g_\nu$ , one has
\eq{F(D^\sharp f)\mathring{h}_{ij} 
	&= F(D^\sharp f)(h_{ij}-\tfrac1nHg_{ij})\\
	&= D^2f(x_i,x_j)-\tfrac1ng^{kl}D^2f(x_k,x_l)g_{ij}\\
	&= \mathring{D}^2f(x_i,x_j)+\tfrac{1}{n+1}\bar\De_Ff\bar g_\nu(x_i,x_j)\\
	&\hp{=}-\frac1ng^{kl}\left(\mathring{D}^2f(x_k,x_l)+\tfrac{1}{n+1}\bar\De_Ff\bar g_\nu(x_k,x_l)\right)g_{ij}\\
	&= \mathring{D}^2f(x_i,x_j)-\tfrac1n g^{kl}\mathring{D}^2f(x_k,x_l)g_{ij}\\
	&= \mathring{D}^2f(x_i,x_j)+\tfrac1n\mathring{D}^2f(\nu,\nu)g_{ij},
}
where we have used
	\eq{g^{kl}x_k^\alpha x_l^\beta=\bar g_\nu^{\alpha\beta}-\nu^\alpha\nu^\beta.}
This gives
	\eq{\label{eq-Sch21-3-1}F^2(D^\sharp f)\vert\mathring{S}\vert^2
	&=\mathring{D}^2f(x_i,x_j)\mathring{D}^2f(x_k,x_l)g^{ik}g^{jl}-\tfrac1n\left(\mathring{D}^2f(\nu,\nu)\right)^2\\
	&\leq \bar g^{\alpha\beta}_\nu\bar g^{\gamma\tau}_\nu\left(f_{,\al\ga}-\tfr{\bar\De_Ff}{n+1}(\bar g_{\nu})_{\al\ga}\right)\left(f_{,\be\tau}-\tfr{\bar\De_Ff}{n+1}(\bar g_{\nu})_{\be\tau}\right)\\
&=\vert \mathring{\bar\n}_{F}^2f\vert^2,
}
where we have used \eqref{eq-bar-g} and \eqref{eq-tracefree-D^2f}.
		
To proceed, we consider the level-set flow
	\eq{\partial_t\phi(t,\xi) 
	&= \frac{D^\sharp F(D^\sharp f(\phi(t,\xi)))}{F(D^\sharp f(\phi(t,\xi)))}
	=\frac{\nu(\phi(t,\xi))}{F(D^\sharp f(\phi(t,\xi)))},\\
	\phi(0,\xi)
	&=\xi,
}
	which makes $\mathcal{U}$ diffeomorphic to $(0,\max_{\bar\cU}f)\times M$ since for any $t>0$ and for any $\xi\in M$, there holds
	\eq{f(\phi(t,\xi)) = \int_0^t\frac{d}{ds}f(\phi(t,\xi))\,ds=t,
}
due to the homogeneity of $F$.
Our aim is to apply \cite[Thm.~1.2]{De-RosaGioffre:/2021} to obtain a stability estimate. To this end, we have to make sure that the $\ti\mu(M_t)$ are comparable to $\ti\mu(\cW)$. Equivalently it suffices to prove the $\mu(M_{t})$ are comparable to $\mu(M)  =1,$ due to \eqref{defn-m_F-M_F} and \eqref{anisotropic area}.
To accomplish this,
note that the evolution of the area element under a hypersurface variation with speed $\partial_t\phi$ is given by (see \cite[(24)]{Xia:/2017})
\eq{\partial_t\,d\mu =\frac{H}{F(D^\sharp f)}\,d\mu = \fr{H}{\abs{D^{\sharp}f}}d\ti\mu,
}
and hence we find with the help of \eqref{condi-hnorm-upperbound},
\eq{
\label{esti-area-flowsurface}\partial_t\mu(M_{t})
	\leq\frac{\sqrt{n}}{\min_{\overline{\mathcal{U}}}\abs{D^{\sharp}f}}\|S\|_{1,M_{t}}&\leq\frac{\sqrt{n}\ti\mu(M_{t})^{\frac{p-1}{p}}}{\min_{\overline{\mathcal{U}}}\abs{D^\sharp f}}\|S\|_{p,M_{t}}\\
	&\leq\frac{C}{\min_{\overline{\mathcal{U}}}\abs{D^\sharp f}}\max(1,\mu(M_{t})),
}
and we may use a similar argument to obtain an estimate from below.
Consequently, recall that $\mu(M)=1$, we obtain
	\eq{1-\frac{C}{\min_{\overline{\mathcal{U}}}\abs{D^\sharp f}}t
	\leq \mu(M_{t})
	\leq e^{\frac{C}{\min_{\overline{\mathcal{U}}}\abs{D^\sharp f}}t}.
}
By defining
	\eq{\label{defn-T1}T_1:=\min\br{\frac{\min_{\overline{\mathcal{U}}}\abs{D^\sharp f}}{2C},\max_{\overline{\mathcal{U}}}f}
}
we deduce that the anisotropic areas of $\{M_t\}_{0\leq t<T_1}$ are uniformly controlled,
	\eq{\tfrac12\leq\vert M_t\vert\leq2,\quad\forall 0\leq t<T_1.
}
	    
For any $t_0\in(0,T_1)$, we integrate \eqref{eq-Sch21-3-1} on $\phi((0,t_0)\times M)$ and use the coarea formula to find
	\eq{\int_0^{t_0}\int_{M_s}\frac{F^{p}(D^\sharp f)\vert \mathring{S}\vert^p}{\vert D^\sharp f\vert}\,d\tilde\mu
	\leq&\int_0^{t_0}\int_{M_s}\frac{\vert\mathring{\bar\n}_{F}^2f\vert^p}{\vert D^\sharp f\vert}\,d\tilde\mu\leq\int_{\mathcal{U}}\vert\mathring{\bar\n}_{F}^2f\vert^p\,dx.
}
From this and \eqref{defn-m_F-M_F}, we obtain
	\eq{\mathcal{L}^1\left(\left\{s\in(0,t_0):\int_{M_s}\vert\mathring{S}\vert^p\,d\tilde\mu
	>\frac{2}{t_0m_F^p\min_{\overline{\mathcal{U}}}\vert D^\sharp f\vert^{p-1}}\int_{\mathcal{U}}\vert\mathring{\bar\n}_{F}^2f\vert^p\, dx\right\}\right)
	\leq\frac{t_0}{2},
}
where $\mathcal{L}^1$ is the one-dimensional Lebesgue measure.
Now define
	\eq{t_0\coloneqq\frac{2}{m_F^p}\min\left(\min_{\overline{\mathcal{U}}}\vert D^\sharp f\vert,\max_{\overline{\mathcal{U}}}f\right)^{\frac{1}{p+1}}\|\mathring{\bar\n}_{F}^2f\|^{\frac{p}{p+1}}_{p,\mathcal{U}}.
}
At this point we have to assume that $t_0<T_1$ and hence in accordance with \eqref{defn-T1}, we demand
	\eq{\label{ineq-tracefree-D^2f}\|\mathring{\bar\n}_{F}^2f\|_{p,\overline{\mathcal{U}}}^{\frac{p}{p+1}}
	\leq \frac{m_F^p}{4C}\min\left(\min_{\overline{\mathcal{U}}}\vert D^\sharp f\vert,\max_{\overline{\mathcal{U}}}f\right)^{\frac{p}{p+1}}.
}
Thanks to the definition of $t_0$, there exists $s\in(0,t_0)$ such that $M_s$ satisfies
	\eq{\label{ineq-S_F-M_s}\|\mathring{S}\|_{p,M_s}
	\leq \frac{1}{\min\left(\min_{\overline{\mathcal{U}}}\vert D^\sharp f\vert,\max_{\overline{\mathcal{U}}}f\right)^{\frac{p}{p+1}}} \|\mathring{\bar\n}_{F}^2f\|_{p,\mathcal{U}}^{\frac{p}{p+1}}.
}
	    
To appeal to \cite[Thm.~1.2]{De-RosaGioffre:/2021}, we rescale $M_s$ by
	\eq{\breve M_s\coloneqq\lambda_sM_s,
}
where
\eq{
(\tfrac{m_F}{2})^{\frac1n}\ti\mu(\cW)^{\fr 1n}
\leq \lambda_s=\left(\frac{\ti\mu(\cW)}{\ti\mu(M_{s})}\right)^{\frac1n}
\leq (2M_F)^{-\frac1n}\ti\mu(\cW)^{\fr 1n}.
}
Here we have used 
\eq{\tfr{1}{M_{F}}\mu(M_{s})\leq \ti\mu(M_{s})\leq \tfr{1}{m_{F}}\mu(M_{s}).} 
Then there holds $\ti\mu(M_s)=\ti\mu(\cW)$ and
	\eq{\label{eq-rescaled-h}
	\|{\breve S}\|_{p, \breve M_s} = \lambda_s^{\frac{n-p}{p}} \|{S}\|_{p,M_s},\q \|\mathring{\breve S}\|_{p,\breve M_s} = \lambda_s^{\frac{n-p}{p}} \|\mathring{S}\|_{p,M_s}.
}
	    
Applying \cite[Thm.~1.2]{De-RosaGioffre:/2021}, we see that there exists positive numbers $\delta_1$ and $C_1$, depending only on $n,F,p,C_0$, such that whenever
	\eq{\label{ineq-delta-1}\|\mathring{\breve S}\|_{p,\breve M_s}\leq\delta_1,
}
then there is a parametrization $\breve\psi_s:\mathcal{W}\ra\breve M_s$ and a vector $v=v(\breve M_s)$ with
	\eq{\|\breve\psi_s-{\rm Id}-v\|_{W^{2,p}(\mathcal{W})}
	\leq C_1\|\mathring{\breve S}\|_{p,\breve M_s}.
}
From \eqref{ineq-tracefree-D^2f}, \eqref{ineq-S_F-M_s} and \eqref{eq-rescaled-h} we conclude that condition \eqref{condi-tracefree-D^2f}  is sufficient to achieve \eqref{ineq-delta-1}.
On the other hand, up to a translation, we may assume without loss of generality that $v=0$.
In particular, since $p>n$, we obtain the estimate on Hausdorff distance:
\eq{\label{ineq-Hausdist-M_s}
	\dist(M_s,\mathcal{W}_{\lambda_s^{-1}})
	&\leq C_1\lambda_s^{\frac{n-p}{p}}\|\mathring{S}\|_{p,M_s}\\
	&\leq \frac{C(n,F,p,C_0)}{\min\left(\min_{\overline{\mathcal{U}}}\vert D^\sharp f\vert,\max_{\overline{\mathcal{U}}}f\right)^{\frac{p}{p+1}}} \|\mathring{\bar\n}_{F}^2f\|_{p,\mathcal{U}}^{\frac{p}{p+1}}.
}
To conclude the proof, note that every $\xi\in M$ is connected to a unique point in $M_s$ by the curve $\phi(\cdot,\xi)$, with
	\eq{
	\xi=\phi(0,\xi)\in M,\quad\phi(s,\xi)\in M_s.
}
Along this curve, we easily find
\eq{\label{esti-dist-flowsurface}
\dist(\xi,\phi(s,\xi))
	\leq \frac{sM_F}{m_F\min_{\overline{\mathcal{U}}}\vert D^\sharp f\vert}
	\leq \frac{t_0M_F}{m_F\min\left(\min_{\overline{\mathcal{U}}}\vert D^\sharp f\vert,\max_{\overline{\mathcal{U}}}f\right)},
} 
it follows from \eqref{ineq-Hausdist-M_s} and the definition of $t_0$ that
	\eq{
	\dist(M,\mathbf{W})
	\leq \frac{C(n,F,p,C_0)}{\min\left(\min_{\overline{\mathcal{U}}}\vert D^\sharp f\vert,\max_{\overline{\mathcal{U}}}f\right)^{\frac{p}{p+1}}} \|\mathring{\bar\n}_{F}^2f\|_{p,\mathcal{U}}^{\frac{p}{p+1}},
}
for some Wulff sphere $\mathbf{W}$. This completes the proof.
\end{proof}

\section{Applications of level-set stability}

\subsection{Preliminary regularity: Existence of tubular neighbourhoods}
\begin{definition}[Uniform interior Wulff sphere condition]\label{Defn-interior-Wulffshape}
Let $F$ be an elliptic integrand and $\Om$ be a bounded domain in $\mathbb{R}^{n+1}$ with boundary $M$.
For $r>0$,
we say that
\textit{$M$ satisfies the uniform interior Wulff shape condition with radius $r$}, if
for every $x\in M$, there exists a Wulff sphere $\mathcal{W}_{r}(y_x)\subset\overline\Om$ centered at $y_x$ with radius $r$, such that $\mathcal{W}_{r}(y_x)\cap M=\{x\}$.
\end{definition}

$f\in W_{0}^{1,2}(\Om)$ is called the \textit{$F$-anisotropic torsion potential of $\Om$} if $f$ is the distributional solution of
\eq{\label{defn-torsion-potential}
    \begin{cases}
        \dive(D^{\sharp}(\tfr 12 F^{2})(D^{\sharp}f)) =\bar\De_Ff=1,\quad&\text{in }\Om,\\
        f=0,\quad&\text{on }M.
    \end{cases}
}

We need some further notation. For a subset $A\sub \bbR^{n+1}$ we define $\abs{A}$ to be the $(n+1)$-dimensional Lebesgue measure. For a function $f$ in the H\"older space $C^{k,\be}(\bar\Om)$ we define $\abs{f}_{k,\be,\bar\Om}$ to be its full H\"older norm, while $\vert M\vert_{2,\al}$ is an abbreviation for control of the geometry of a hypersurface $M$ in $C^{2,\al}$-sense, as it appears in the constant of the Schauder estimate. We also write $\abs{M}_{2}:=\abs{M}_{2,0}$. We start with the following observation.

\begin{lemma}
Given an elliptic integrand $F$ and a bounded domain $\Om\subset\mathbb{R}^{n+1}$ with $C^{2}$-boundary $M$, let $f$ be the $F$-anisotropic torsion potential of $\Om$.
Then
\eq{\label{ineq-C0bound-f}
    \vert f\vert_{C^0(\overline\Om)}
    \leq\frac{1}{2(n+1)}\left(\frac{\vert\Om\vert}{\ti\mu(\cW)}\right)^{\frac{2}{n+1}}.
}
\end{lemma}
\begin{proof}
This follows from a standard PDE comparison argument.	
Note that, since $\mathbb{R}^{n+1}$ is clearly a convex cone with no boundary,
we may apply \cite[Corollary 4.4]{LuXiaZhang:04/2023} with $F_\theta$ therein replaced by a general elliptic integrand $F$ to get the assertion.
\end{proof}

Now we define the critical set of $f$ as
\eq{
    \mathcal{C}=\{x\in\overline\Om: D^\sharp f(x)=0\},
}
and we have the following regularity properties of $f$:

\begin{prop}\label{torsion regularity}
Given an elliptic integrand $F$ and a bounded domain $\Om\subset\mathbb{R}^{n+1}$ with  $C^{2}$-boundary $M$. Let $f$ be the $F$-anisotropic torsion potential of $\Om$.
There holds 
\enum{
\item $\abs{\mathcal{C}}=0$;
\item $f<0$ in $\Om$;
\item $f\in C^{1,\beta}(\overline\Om)\cap W^{2,2}(\Om)$ for some $\beta = \be(n,F)\in(0,1)$
and there holds
\eq{
    \vert f\vert_{1,\beta,\overline\Om}\leq C(n,F,\vert\Om\vert,\vert M\vert_2).
}
}
\end{prop}

\pf{
(i) and (ii) are content of \cite[Prop.~3.2]{DelgadinoMaggiMihailaNeumayer:07/2018}, while (iii) follows from \cite[Prop.~2.3]{CianchiSalani:12/2009} and \cite{Lieberman:/1988}.
Indeed, the Schauder estimate \cite[Thm.~1]{Lieberman:/1988} (choosing $\alpha=1,m=0$ therein) in conjunction with the $C^0$-bound \eqref{ineq-C0bound-f} gives that
\eq{\label{esti-Schauder}
    \vert f\vert_{1,\beta,\overline\Om}
    \leq C(n,F,\vert\Om\vert,\vert M\vert_2),
}
where the dependence on $\vert M\vert_2$ is due to the fact that the Schauder estimate \cite[Thm.~1]{Lieberman:/1988} is built on flattening the boundary and using the local estimates near the flattened boundary \cite[Lemma 1-4]{Lieberman:/1988}. Here we point out that $\beta$ only depends on $n,F$.
}

\begin{lemma}[Gradient bound on $M$]\label{Lem-IWC-bound}
Let $\Om\subset\mathbb{R}^{n+1}$ be a bounded domain with  $C^{2}$-boundary $M$ that satisfies the uniform interior Wulff sphere condition with radius $r$ and let $f\in C^{1,\beta}(\overline\Om)\cap W^{2,2}(\Om)$ for some $\beta\in(0,1)$ be a solution of \eqref{defn-torsion-potential} in $\Om$,
then
\eq{
\vert D^\sharp f\vert
    \geq\frac{r}{(n+1)M_F}\quad\text{on }M.
}
\end{lemma}

\begin{proof}
For any $p\in M$, let $\mathcal{W}_{r}(y_p)$ be the interior Wulff sphere touching $M$ at $p$.
Using \eqref{F and F0}, we see that the function 
\eq{w(x)\coloneqq\frac{F^0(x-y_p)^2-r^2}{2(n+1)}} is the $F$-anisotropic torsion potential of the Wulff ball centered at $y_p$ with radius $r$.
Indeed, there holds
\eq{\label{reference function}
D^\sharp w(x)
  =\frac{1}{n+1}F^0(x-y_p)D^\sharp F^0(x-y_p),\quad\bar\nabla^\sharp_Fw(x)
  =\frac{x-y_p}{n+1},
}
and hence, using \eqref{eq: Q}, $w$ solves
\eq{
\begin{cases}
\bar\De_Fw=1\quad&\text{in }\Om,\\
    w\geq0,\quad&\text{on }M.
\end{cases}
}
In particular,
by a standard comparison as given in \cite[equ.~(3.9)]{DelgadinoMaggiMihailaNeumayer:07/2018}, we have $w\geq f$ on $\overline\Om$, so that $w(p)=f(p)=0$ implies
\eq{\frac{\partial f}{\partial\tilde \nu}(p)
    \geq \frac{\partial w}{\partial\tilde\nu}(p)
    &=\frac{1}{n+1}F^0(p-y_p)\left<D^\sharp F^0(p-y_p),\tilde\nu\right>\\
    &=\frac{1}{n+1}F^0(p-y_p)\left<D^\sharp F^0(\nu),\tilde\nu\right>\\
    &=\frac{r}{n+1}\frac{1}{F(\tilde\nu(p))}\\
    &\geq\frac{r}{(n+1)M_F},
}
where we have used the fact that
$p-y_p=r\nu(p)=rD^\sharp F(\tilde\nu(p))$ for the first equality since $\mathcal{W}_{r}(y_p)$ touches $M$ from the interior at $p$, and \eqref{F and F0} for the last equality.
The assertion follows.
\end{proof}

We conclude the following uniform one-sided neighbourhood.

\begin{prop}\label{Prop-DMMN18-Prop-3-2}
Let $F$ be an elliptic integrand and $\Om\subset\mathbb{R}^{n+1}$ a bounded domain with $C^{2}$-boundary $M$ that satisfies the uniform interior Wulff sphere condition with radius $r$. Let $f$ be the $F$-anisotropic torsion potential of $\Om$.
Then there exists a one-sided neighborhood $\cU_{f}$ of $M$, such that $\mathcal{U}_f\cap\mathcal{C}=\emptyset$, with 
\eq{
\overline{\mathcal{U}}_f
    =\bigcup_{\min_{\overline{\mathcal{U}}_f}f\leq t\leq0}\{x\in\overline\Om:f(x)=t\}
}
and 
\eq{C^{-1}
\leq\min\left(\mu(M)^{\fr 1n}\min_{\overline{\mathcal{U}}_f}\vert D^\sharp f\vert,-\min_{\overline{\mathcal{U}}_f}f\right)
}
for some positive constant $C$ that depends on $n,F,r,\mu(M),\vert M\vert_2$. In case that $M\in C^{2,\al}$ for some $0<\al<1$, we obtain
\eq{\abs{f}_{2,\al,\bar\cU_{f}}\leq C,}
where $C$ additionally depends on $\al$ and $\abs{M}_{2,\al}$.
\end{prop}

\begin{proof}
By virtue of \autoref{Lem-IWC-bound}, $\vert D^\sharp f\vert$ has a uniform positive lower bound on $M$, depending only on $n,F,r$, and hence $\dist(M,\mathcal{C})>0$ strictly.
Indeed, the positive lower bound of $\vert D^\sharp f\vert_{|M}$ and the uniform upper bound of the H\"older-norm of $\vert D^\sharp f\vert$ ensure the existence of the one-sided neighborhood of $M$, say $\mathcal{U}_f$, such that $\dist(\mathcal{U}_f,\mathcal{C})>0$ and $\overline{\mathcal{U}}_f$ is foliated by the level-sets
\eq{
\overline{\mathcal{U}}_f
    =\bigcup_{\min_{\overline{\mathcal{U}}_f}f\leq t\leq0}\{x\in\overline\Om:f(x)=t\},
}
on which $f$ is smooth and satisfying
\eq{
\label{esti-minf-mingradf}
    C^{-1}\leq\min\left(\mu(M)^{\fr 1n}\min_{\overline{\mathcal{U}}_f}\vert D^\sharp f\vert,-\min_{\overline{\mathcal{U}}_f}f\right)
}
for some positive constant $C$ that depends on $n,F,r,\mu(M),\vert M\vert_2$. The final assertion follows from standard Schauder theory, as the operator is uniformly elliptic in $\bar\cU_{f}$.
\end{proof}

\subsection{Stability of the anisotropic Heintze-Karcher inequality}
Recall that the anisotropic Heintze-Karcher (see \cite[Thm.~4.4]{GeHeLiMa:/2009} and also \cite{JiaWangXiaZhang:04/2023} for a refined version) states: for a closed compact embedded strictly anisotropic mean-convex $C^2$-hypersurface $M\subset\mathbb{R}^{n+1}$, there holds that
\eq{\label{ineq-HK}
\int_M\frac{1}{H}\,d\mu
    -\frac{n+1}{n}\vert\Om\vert\geq0,
}
with equality holds if and only if $M$ is a Wulff sphere. In this section we prove corresponding stability result \autoref{Thm-Sta-HK}.

\begin{prop}
Given an elliptic integrand $F$ and a bounded domain $\Om\subset\mathbb{R}^{n+1}$ with $C^{2}$-boundary $M$ such that $H>0$ along $M$, let $f$ be the $F$-anisotropic torsion potential of $\Om$.
Then there holds
\eq{\label{eq-integral-identity-HK}
    \int_\Om \vert\mathring{\bar\n}_{F}^2f\vert^2\, dx
    \leq \br{\frac{n}{n+1}}^2\left(\int_{M}\frac{1}{H}\,d \mu-\frac{n+1}{n}\vert\Om\vert\right).
}
\end{prop}
\begin{proof}
The anisotropic Reilly's formula \cite[Proposition 3.3]{DelgadinoMaggiMihailaNeumayer:07/2018} together with \eqref{eq-tracefree-D^2f} gives
\eq{\label{eq-Reilly}
    \int_\Om \vert\mathring{\bar\n}_{F}^2f\vert^2\,dx
    =\frac{n}{n+1}\vert\Om\vert-\int_{M}HF^2(D^\sharp f)\,d\mu.
}
On the other hand, by using the divergence theorem \cite[Lemma 4.3]{CiraoloLi:08/2022} and \eqref{F-homogeneity},
\eq{\label{eq-Omega}
    \vert\Om\vert
    =\int_\Om{\rm div}(D^{\sharp}(\tfr 12 F^{2})(D^{\sharp}f))\,d x
    =\int_{M}F(D^\sharp f)\,d\mu,
}
and it follows from the H\"older's inequality that
\eq{\vert\Om\vert^2
    =\left(\int_{M}F(D^\sharp f)\, d\mu\right)^2
    \leq\left(\int_{M}HF^2(D^\sharp f)\,d\mu\right)\left(\int_{M}\frac{1}{H}\,d\mu\right),
}
which in turn implies
\eq{
    \int_\Om \vert\mathring{\bar\n}_{F}^2f\vert^2\,d x
    \leq \frac{n}{n+1}\vert\Om\vert-\frac{\vert\Om\vert^2}{\int_{M}\frac{1}{H}\,d\mu}
    &=\frac{n}{n+1}\vert\Om\vert\left(\frac{\int_{M}\frac{1}{H}\,d\mu-\frac{n+1}{n}\vert\Om\vert}{\int_{M}\frac{1}{H}\,d\mu}\right)\\
    &\leq\br{\frac{n}{n+1}}^2\left(\int_{M}\frac{1}{H}\,d\mu-\frac{n+1}{n}\vert\Om\vert\right),
    }
    where we have used \eqref{ineq-HK} to derive the last inequality. 
\end{proof}

We can now prove the stability result.

\begin{proof}[Proof of \autoref{Thm-Sta-HK}]
We use \eqref{eq-integral-identity-HK} and we want to appeal to the level-set stability result \autoref{Thm-levelset-stab} with $p=n+1$. Let $f$ be the $F$-anisotropic torsion potential. Then $\cU_{f}$ as constructed in \autoref{Prop-DMMN18-Prop-3-2} satisfies, up to a negligible sign, the assumptions of \autoref{Thm-levelset-stab}. Furthermore in this neighbourhood we have $C^{2,\al}$-estimates from \autoref{Prop-DMMN18-Prop-3-2}. 
Now we have to verify \eqref{condi-hnorm-upperbound}.
Together with \eqref{eq-integral-identity-HK} we obtain
\eq{
\label{esti-HKdeficit}
    \int_{\mathcal{U}_f}\vert\mathring{\bar\n}_{F}^2f\vert^{n+1}\,d x \leq C\left(\int_M\frac{1}{H}\,d\mu-\frac{n+1}{n}\vert\Om\vert\right)
}
for some positive constant $C$ depends on $n,\al,F,r,\mu(M)\vert,\vert M\vert_{2,\al}$.
Taking \eqref{esti-minf-mingradf} into account, we may get the desired control in \eqref{condi-tracefree-D^2f},
once we provide a smallness condition on the quantity in \eqref{esti-HKdeficit}. The assertion then follows from \autoref{Thm-levelset-stab}.
Precisely, we conclude from \eqref{conclu-distance} that
\eq{
\dist(M,{\bf W})
    \leq C\left(\int_{M}\frac{1}{H}\,d\mu-\frac{n+1}{n}\vert\Om\vert\right)^{\frac{1}{n+2}}
}
for some Wulff sphere ${\bf W}\subset\mathbb{R}^{n+1}$ and some positive constant $C$ that depends on the same factors as above. 
\end{proof}
\subsection{Stability of the anisotropic Alexandrov theorem}

Recall that given 
an elliptic integrand $F$ and a bounded domain $\Om\subset\mathbb{R}^{n+1}$ with smooth boundary $M$, if the anisotropic mean curvature of $M$ is constant, then
\eq{ \frac{n}{H}\int_{M}\,d\mu
    =\int_{M}\left<x,\tilde\nu\right>\,d\tilde\mu
    =(n+1)\vert\Om\vert,
}
where we have used the anisotropic Minkowski formula (see \cite{HeLi:04/2008}) in the first equality.
Enlightened by this, we define the \textit{reference anisotropic mean curvature of $\Om$} as
\eq{\mathfrak{H}(\Om)
    \coloneqq\frac{n}{n+1}\frac{\mu(M)}{\vert\Om\vert},
}
and correspondingly, the \textit{reference anisotropic radius of $\Om$} is defined as $R_1(\Om)=\frac{n}{\mathfrak{H}(\Om)}$.
For simplicity, we shall suppress the dependence on $\Om$ of these quantities in all follows.

\begin{prop}
    Given an elliptic integrand $F$ and a bounded domain $\Om\subset\mathbb{R}^{n+1}$ with $C^{2}$-boundary $M$, let $f$ be the $F$-anisotropic torsion potential of $\Om$.
    Let $\mathfrak{H}$, $R_1$ be the anisotropic reference mean curvature and radius of $\Om$, then there holds
    \eq{\label{eq-integral-identity-Alex}    &\int_\Om\vert\mathring{\bar\n}_{F}^2f\vert^2\,d x+\frac{n}{n+1}\frac{1}{\tilde R_1}\int_{M}\left(F(D^\sharp f)-\tilde R_1\right)^2\,d \mu\\
    =& \int_{M}(\mathfrak{H}-H)F^2(D^\sharp f)\,d\mu,
    }
    where $\tilde R_1
    \coloneqq\frac{R_1}{n+1}
    = \frac{n}{n+1}\frac{1}{\mathfrak{H}}$.
\end{prop}

\begin{proof}
Recall that from \eqref{eq-Omega} a simple computation yields 
\eq{
&\frac{1}{\tilde R_1}\int_{M}F^2(D^\sharp f)\,d\mu\notag\\
    =~&\fr{1}{\ti R_{1}}\int_{M}\left(F(D^\sharp f)-\tilde R_1\right)^2\,d \mu+2\int_{M}F(D^\sharp f)\,d\mu-\tilde R_1\mu(M)\notag\\
    =~&\frac{1}{\tilde R_1}\int_{M}\left(F(D^\sharp f)-\tilde R_1\right)^2\,d \mu+2\vert\Om\vert-\tilde R_1\mu(M)\notag\\
    =~&\frac{1}{\tilde R_1}\int_{M}\left(F(D^\sharp f)-\tilde R_1\right)^2\,d \mu+\vert\Om\vert,
}
where we have used the definition of $\tilde R_1$ for the last equality.
Since $\mathfrak{H}=\frac{n}{n+1}\frac{1}{\tilde R_1}$ is a constant, we may use the above equality to find
\eq{
\int_{M}HF^2(D^\sharp f)\,d\mu
    &=\mathfrak{H}\int_{M}F^2(D^\sharp f)\,d\mu+\int_{M}(H-\mathfrak{H})F^2(D^\sharp f)\,d\mu\\
    &=\frac{n}{n+1}\left(\frac{1}{\tilde R_1}\int_{M}\left(F(D^\sharp f)-\tilde R_1\right)^2\,d \mu+\vert\Om\vert\right)\\
    &\hp{=}+\int_{M}(H-\mathfrak{H})F^2(D^\sharp f)\,d\mu.
}
Putting this back into \eqref{eq-Reilly}, we get \eqref{eq-integral-identity-Alex}.
\end{proof}

\begin{proof}[Proof of \autoref{Thm-Sta-Alex}]
From \eqref{eq-integral-identity-Alex} we obtain
\eq{
\int_{\mathcal{U}_f}\vert\mathring{\bar\n}_{F}^2f\vert^2\,d x
\leq&\int_{M}(\mathfrak{H}-H)F^2(D^\sharp f)\,d\mu\leq M^3_F\max_{M}\vert D^\sharp f\vert^2\left\|H-\mathfrak{H}\right\|_{1,M}.
}
The rest of the proof follows from that of \autoref{Thm-Sta-HK}.
\end{proof}

\subsection{Stability of the anisotropic overdetermined problem}
Consider the \textit{anisotropic $P$-function} of the $F$-anisotropic torsion potential $f$, defined as
\eq{
P\coloneqq\frac{F^2(D^\sharp f)}{2}-\frac{1}{n+1}f,
}
on $\Om\bs\mathcal{C}$,
recall that $\nu(x)=D^\sharp F(D^\sharp f(x))$ and
the $F$-gradient of $P$ is then given by
\eq{
\bar\nabla_F^\sharp P(x)
= \bar g_{\nu(x)}^{\alpha\beta}P_{,\beta}(x)e_\alpha.
}

From \cite[equ.~(2.8)]{Xia:/2012} we know
\eq{\label{eq-div(A-gradient-P)}
    {\rm div}(\bar\nabla_F^\sharp P)
    = \vert \mathring{\bar\n}_{F}^2f\vert^2.
    }

We record the Pohozaev-type identity for $P$, which can be deduced as in \cite[Thm.~5]{WangXia:/2011}.

\begin{prop}
Given an elliptic integrand $F$ and a bounded domain $\Om\subset\mathbb{R}^{n+1}$ with $C^{2}$-boundary $M$, let $f$ be the $F$-anisotropic torsion potential of $\Om$.
Then there holds   \eq{\label{eq-Pohozaev}
    \int_\Om P\,d x
    = \frac{1}{2(n+1)}\int_{M}F^2(D^\sharp f)\left<x,\tilde\nu\right>\,d\tilde\mu.
} 
\end{prop}
\begin{proof}
Thanks to \autoref{torsion regularity},     
we may use the divergence theorem \cite[Lemma~4.3]{CiraoloLi:08/2022} for the vector-field $xF^2(D^\sharp f(x))$ to obtain
\eq{
\int_{M}F^2(D^\sharp f)\left<x,\tilde\nu\right>\,d\ti\mu
    &=\int_\Om{\rm div}(xF^2(D^\sharp f))\,d x\\
    &=(n+1)\int_\Om F^2(D^\sharp f)\,d x+2\int_\Om \left<x,D^2f[\bar\nabla^\sharp_Ff]\right>\,d x.
}
On the other hand, recall that $f$ solves \eqref{defn-torsion-potential}, one has
\eq{
-(n+1)\int_\Om f\,d x
    &=\int_\Om ({\rm div}(fx)-(n+1)f)\,d x\\
    &=\int_\Om\left<x,D^\sharp f\right>{\rm div}(\bar\nabla^\sharp_Ff)\,d x\\
    &=\int_\Om{\rm div}\left(\left<x,D^\sharp f\right>\bar\nabla^\sharp_Ff\right)-\left<\bar\nabla^\sharp_Ff,D^\sharp f+D^2f[x]\right>\,d x\\
    &=\int_{M}F^2(D^\sharp f)\left<x,\tilde\nu\right>\,d\tilde\mu\\
    	&\hp{=}-\int_\Om F^2(D^\sharp f)+\left<x,D^2f[\bar\nabla^\sharp_Ff]\right>\,d x.
}
Using integration by parts, one finds
\eq{\label{eq-WX11-(38)}
    0=\int_\Om{\rm div}(f\bar\nabla^\sharp_Ff)\,d x
    =\int_{\Om}F^2(D^\sharp f)\,d x+\int_\Om f\,d x.
}
Combining these equalities, we see that
\eq{
&\int_{M}F^2(D^\sharp f)\left<x,\tilde\nu\right>\,d\tilde\mu\\
    =~&2\left(\int_{M}F^2(D^\sharp f)\left<x,\tilde\nu\right>\,d\tilde\mu-\int_\Om F^2(D^\sharp f)\,d x+(n+1)\int_\Om f\,d x\right)\\
    	&+(n+1)\int_\Om F^2(D^\sharp f)\,d x,
}
and hence from \eqref{eq-WX11-(38)}
\eq{
\int_{M}F^2(D^\sharp f)\left<x,\tilde\nu\right>\,d\tilde\mu
    &=(1-n)\int_\Om F^2(D^\sharp f)\,d x-2(n+1)\int_\Om f\,d x\\
    &=(n+1)\int_\Om F^2(D^\sharp f)\,d x-2\int_\Om f\,d x\\
    &=2(n+1)\int_\Om P\,d x.}
\end{proof}

Now we introduce an auxiliary function 
\eq{l(x)=\frac{F^0(x)^2}{2(n+1)}.} 
As in \eqref{reference function} we obtain
\eq{
D^\sharp l(x)=\frac{1}{n+1}F^0(x)D^\sharp F^0(x),\quad\bar\nabla^\sharp_Fl(x)=\frac{x}{n+1},\quad\bar\De_Fl=1,
}
and hence we may rewrite  \eqref{eq-Pohozaev} as
\eq{\label{eq-Pohozaev-2}
    \int_\Om P\,d x
    =\frac{1}{2}\int_{M}F^2(D^\sharp f)\left<\bar\nabla^\sharp_Fl,\tilde\nu\right>\,d\tilde\mu.}
    
\begin{prop}
Given an elliptic integrand $F$ and a bounded domain $\Om\subset\mathbb{R}^{n+1}$ with $C^{2}$-boundary $M$, let $f$ be the $F$-anisotropic torsion potential of $\Om$.
For any $R\in\mathbb{R}$,
there holds
\eq{\label{eq-integral-identity-OverDetermined}
    \int_\Om (-f)\vert\mathring{\bar\n}_{F}^2f\vert^2\,d x
    =\frac12\int_{M}(F^2(D^\sharp f)-R^2)\left<\bar\nabla^\sharp_Ff-\bar\nabla^\sharp_Fl,\tilde\nu\right>\,d\tilde\mu.}
\end{prop}

\begin{proof}
Using integration by parts, we obtain the Green-type identity
\eq{
    \int_\Om f{\rm div}(\bar\nabla_F^\sharp P)-P{\rm div}(\bar\nabla_F^\sharp f)\,d x
    =\int_{M}f\left<\bar\nabla_F^\sharp P,\tilde\nu\right>-P\left<\bar\nabla_F^\sharp f,\tilde\nu\right>\,d\tilde\mu.
}
Since $f$ solves \eqref{defn-torsion-potential}, we may use \eqref{eq-div(A-gradient-P)} and expand $P$ on $M$ to obtain
\eq{
\int_\Om (-f)\vert\mathring{\bar\n}_{F}^2f\vert^2\,d x
    =-\int_\Om P\,d x+\frac{1}{2}\int_{M}F^2(D^\sharp f)\left<\bar\nabla^\sharp_Ff,\tilde\nu\right>\,d\tilde\mu.
}
By virtue of the Pohozaev-type identity \eqref{eq-Pohozaev-2}, this reads
\eq{
    \int_\Om (-f)\vert\mathring{\bar\n}_{F}^2f\vert^2\,d x
    =\frac12\int_{M}F^2(D^\sharp f)\left<\bar\nabla^\sharp_Ff-\bar\nabla^\sharp_Fl,\tilde\nu\right>\,d\tilde\mu.
}
The proof is completed once we notice that
\eq{
    0=\int_\Om{\rm div}(\bar\nabla^\sharp_Ff-\bar\nabla^\sharp_Fl)\,d x
    =\int_{M}\left<\bar\nabla^\sharp_Ff-\bar\nabla^\sharp_Fl,\tilde\nu\right>\,d\tilde\mu.
}
\end{proof}

\begin{proof}[Proof of \autoref{Thm-Sta-OverD}]

We may assume that $0\in\Om$. Moreover, we suppose 
\eq{\left\|F(D^\sharp f)-\tfrac{\vert\Om\vert}{\mu(M)}\right\|_{1,M}>0.}
Otherwise we see from \eqref{eq-integral-identity-OverDetermined}
that $D^2f$ is umbilical with respect to $\bar g_\nu$ and hence $M$ is umbilical in the anisotropic sense thanks to \eqref{eq-Sch21-3-1}.
From \cite{GeHeLiMa:/2009} we conclude that $M$ must be a Wulff sphere.

Our starting point is that, if $F(D^\sharp f)$ is constant along $M$, then \eqref{eq-Omega} shows that the constant is in fact characterized by $\abs{\Om}/\mu(M)$.
To proceed, we choose $R=\abs{\Om}/\mu(M)$ in \eqref{eq-integral-identity-OverDetermined}, which gives
\eq{
\int_\Om(-f)\vert\mathring{\bar\n}_{F}^2f\vert^2\,d x
    \leq C\left\|F(D^\sharp f)-\tfrac{\vert\Om\vert}{\mu(M)}\right\|_{1;M},
}
for some positive constant $C$ depending on $n,F,\max_{M}\vert D^\sharp f\vert,\diam(\Om)$,
where we have used 
\eq{R=\frac{\vert\Om\vert}{\mu(M)} \leq C(n,\diam(\Om)),}
which follows from the isodiametric and isoperimetric inequalities,
and 
\eq{\max_{ M}\left<\bar\nabla^\sharp_Fl,\tilde\nu\right>
\leq C(n,\diam(\Om)).}

From \autoref{Prop-DMMN18-Prop-3-2} we obtain $\mathcal{U}_f$, the one-sided neighborhood of $M$, on which all the crucial quantities are under control.
In particular, as in \eqref{esti-minf-mingradf}, there exists a positive constant, denoted by $\hat C$ in the following, which depends on $n,F,r,\vert M\vert_2$, such that
\eq{
\label{esti-minf-mingradf-OP}
    \hat C^{-1}
    \leq\min\left(\mu(M)^{\fr 1n}\min_{\overline{\mathcal{U}}_f}\vert D^\sharp f\vert,-\min_{\overline{\mathcal{U}}_f}f\right).
}
We require that
\eq{
\label{esti-epsilon}
\ep&\coloneqq\left\|F(D^\sharp f)-\tfrac{\vert\Om\vert}{\mu(M)}\right\|_{1,M}^{\tfrac12}\leq \fr{1}{2\hat C}.
}
Hence the set $\mathcal{U}_\ep\coloneqq\mathcal{U}_f\cap\{x\in\Om:f(x)<-\ep\}$ is non-empty and foliated by
\eq{
\overline{\mathcal{U}_\ep}
    =\bigcup_{\min_{\overline{\mathcal{U}}}f\leq t\leq-\ep}\{x\in\Om:f(x)=t\}.
}
Moreover, we have
\eq{
\ep\int_{\mathcal{U}_\ep}\vert\mathring{\bar\n}_{F}^2f\vert^2\,d x
    &\leq\int_{\mathcal{U}_\ep}(-f)\vert\mathring{\bar\n}_{F}^2f\vert^2\,d x\\
    &\leq C(n,F,\max_{x\in M}\vert D^\sharp f\vert,\diam(\Om))\left\|F(D^\sharp f)-\tfrac{\vert\Om\vert}{\mu(M)}\right\|_{1,M}\\
    &=C(n,F,\max_{x\in M}\vert D^\sharp f\vert,\diam(\Om))\ep^2,
}
so that
\eq{
\label{esti-U_ep}
    \int_{\mathcal{U}_\ep}\vert\mathring{\bar\n}_{F}^2f\vert^{n+1}\,d x
    &\leq\max_{\overline{\mathcal{U}}_\ep}\vert\mathring{\bar\n}_{F}^2f\vert^{n-1}\int_{\mathcal{U}_\ep}\vert\mathring{\bar\n}_{F}^2f\vert^2\,d x\\
    &\leq C(n,F,\diam(\Om),\abs{f}_{2,0,\bar\cU_{f}})\ep.
}
We want to conclude the proof by appealing to the level-set stability result \autoref{Thm-levelset-stab}. To this end, we consider $\hat f=f+\ep$, so that $\mathcal{U}_\ep$ is foliated by
\eq{
\overline{\mathcal{U}}_\ep
=\bigcup_{\min_{\overline{\mathcal{U}}}\hat f-\ep\leq t\leq0}\{x\in\Om:\hat f(x)=t\}.
}
We obtain
\eq{
\min\left(\mu(M)^{\fr 1n}\min_{\overline{\mathcal{U}}_\ep}\vert D^\sharp\hat f\vert,-\min_{\overline{\mathcal{U}}_\ep}\hat f\right)
      &\geq\min\left(\mu(M)^{\fr 1n}\min_{\overline{\mathcal{U}}_f}\vert D^\sharp f\vert,-\min_{\overline{\mathcal{U}}_f} f\right)-\ep\\
      &\geq (2\hat C)^{-1}.
}
We are now in the position to apply \autoref{Thm-levelset-stab} (with $p=n+1,\hat f$, and $\mathcal{U}_\ep$),
notice that by virtue of \eqref{esti-U_ep}, \eqref{esti-epsilon} and \eqref{esti-minf-mingradf-OP}, condition \eqref{condi-tracefree-D^2f} will be satisfied once provided a smallness condition on the quantity $\ep$ that depends only on $n,F,r,\vert M\vert_2,\diam(\Om)$ and $\abs{f}_{2,0,\bar\cU_{f}}$, therefore we conclude from \eqref{conclu-distance} that
\eq{
\dist(\{f=-\ep\},{\bf W})
    \leq C\left\|F(D^\sharp f)-\tfrac{\vert\Om\vert}{\mu(M)}\right\|_{1,M}^{\frac{1}{2(n+2)}}.
}
Recall that as shown in the proof of \autoref{Thm-levelset-stab} (see in particular \eqref{esti-area-flowsurface} and the first inequality in \eqref{esti-dist-flowsurface}), the anisotropic area of the flow hypersurface $\{f=-\ep\}$ can be controlled by $\mu(M)$, and the Hausdorff distances between the flow hypersurfaces and the initial hypersurface are also under control:
\eq{
\dist(M,\{f=-\ep\})
    \leq C(n,F,r,\diam(\Om),\vert M\vert_2,\mu(M))\ep,
}
which yields the assertion since $\ep+\ep^{\frac{1}{(n+2)}}\leq2\ep^{\frac{1}{(n+2)}}$ provided $\ep<1$. 
\end{proof}

\begin{rem}
The $C^{2,\al}$-assumption in the main theorems is only used to control the $C^{2}$-norm of $f$. If such control is available by other means, this assumption may be replaced or dropped.
\end{rem}


\providecommand{\bysame}{\leavevmode\hbox to3em{\hrulefill}\thinspace}
\providecommand{\MR}{\relax\ifhmode\unskip\space\fi MR }
\providecommand{\MRhref}[2]{%
  \href{http://www.ams.org/mathscinet-getitem?mr=#1}{#2}
}
\providecommand{\href}[2]{#2}

\end{document}